\documentclass[11pt]{amsart}
   \usepackage{ amsmath,amssymb,verbatim,graphicx,amsthm}
\usepackage[top=1in, bottom=1in, left=1in, right=1in]{geometry}
\usepackage{cite}
 \usepackage{url}

 \allowdisplaybreaks[1]

\usepackage[sc]{mathpazo}
\usepackage[T1]{fontenc}

\newcommand{\vertiii}[1]{{\left\vert\kern-0.25ex\left\vert\kern-0.25ex\left\vert #1
    \right\vert\kern-0.25ex\right\vert\kern-0.25ex\right\vert}}

\theoremstyle{plain}
\newtheorem{theorem}{Theorem}[section]

\newtheorem{lemma}[theorem]{Lemma}

\newtheorem{mydef}[theorem]{Definition}

   \renewcommand{\bf}{\bfseries}

   \newcommand{\T}{\mathbb{T}}
   \newcommand{\Z}{\mathbb{Z}}
   \newcommand{\N}{\mathbb{N}}

   \renewcommand{\epsilon}{\varepsilon}

   \newcommand{\dis}{\displaystyle}

\theoremstyle{plain}

\usepackage{setspace}

\begin{document}
 \title{Pointwise recurrence for commuting measure preserving transformations}
   \author{I. Assani}
\thanks{Department of Mathematics, UNC Chapel Hill, NC 27599, assani@email.unc.edu}
 \begin{abstract}
 Let $(X,\mathcal{A}, \mu)$ be a probability measure space and let $T_i,$ $1\leq i\leq H,$ be commuting invertible measure preserving transformations on this measure space. We prove the following pointwise results; \\
 The averages
 $$\frac{1}{N}\sum_{n=1}^N f_1(T_1^nx)f_2(T_2^nx)\cdots f_H(T_H^nx)$$ converge a.e. for every function $f_i \in L^{\infty}(\mu)$ .\\
%\begin{enumerate}
% \item for $H=2$ and one of the maps $T_1\circ T_2^{-1}$ or $T_2\circ T_1^{-1}$ is ergodic.
% \item for any $H\geq 2$ and the maps $T_i$ are weakly mixing.
 %$\Phi: X^H \rightarrow X^H $ defined as $\Phi(z_1, z_2, ..., z_H) = (T_1(z_1), T_2(z_2),..., T_H(z_H)) $ is minimal.
%\end{enumerate}
As a consequence  if $T_i = T^i$ for $1\leq i \leq H$ where $T$ is an invertible measure preserving transformation on $(X, \mathcal{A}, \mu)$ then the averages
  $$\frac{1}{N}\sum_{n=1}^N f_1(T^nx)f_2(T^{2n}x)...f_H(T^{Hn}x)$$ converge a.e.
  This solves a long open question on the pointwise convergence of nonconventional ergodic averages. For $H=2$ it provides another proof of J. Bourgain's a.e. double recurrence theorem.
\end{abstract} \maketitle

  \section{Introduction}
Let $(X, \mathcal{A}, \mu)$ be a probability measure space that we assume without loss of generality to be atomless.
%Our goal in this paper is to prove the following theorem.
 %\begin{theorem}\label{MET} Let $(X, \mathcal{A}, \mu)$ be a probability measure space and let $H$ be a positive integer. Let $T_i,$ $1\leq i\leq H,$ be invertible measure preserving transformations on this measure space.For every bounded functions $f_i,$ $1\leq i\leq H$ the averages $\dis \frac{1}{N}\sum_{n=1}^N \prod_{i=1}^H f_i(T_i^nx)$ converge a.e. when $N$ tends to $\infty.$
%\end{theorem}
% \{bd Historical backgroung}
%We will assume without loss of generality that it is the unit interval, with Lebesgue measure.
Let $T_i,$ $1\leq i\leq H,$ be invertible measure preserving transformations on this measure space. For $f_i\in L^{\infty}(\mu),$ $1\leq i \leq H,$ we look at
  the well known open problem of the pointwise convergence of the nonconventional ergodic averages $$\frac{1}{N}\sum_{n=1}^N \prod_{i=1}^H f_i(T_i^nx).$$ The case $H=1$ corresponds to the classical ergodic averages for which the pointwise convergence is known by Birkhoff ergodic theorem. In \cite{Fu1}, H. Furstenberg asked if for a measure preserving transformation $T$ on $(X, \mathcal{A}, \mu),$ bounded functions $f,g$ and $m$ a positive integer $m\neq 1$ the averages $\dis \frac{1}{N}\sum_{n=1}^N f(T^nx)g(T^{mn}x)$ converge a.e. J. Bourgain \cite{Bourg} proved that this was indeed the case. The natural question then became; for any positive integer $H,$ and bounded functions $f_1, f_2,...,f_H$ do we have the pointwise convergence of the averages $\dis \frac{1}{N}\sum_{n=1}^N f_1(T^nx)\cdots f_H(T^{Hn}x)?$ Partial results were obtained in \cite{DL} for K-systems and in \cite{Ass1} for weakly mixing systems $T$ for which  the restriction to the Pinsker algebra has singular spectrum. The arguments in this last paper relied in part on J. Bourgain result \cite{Bourg}. We provided a simplification of Bourgain's proof for a class of ergodic dynamical systems in \cite{Ass3} and gave some consequences of this simplification in \cite{Ass2}.  In \cite{HSY} W. Huang, S. Shao and X. Ye announced a positive solution when $T$ is distal. \\
 For the more general case of commuting measure preserving transformations $T_i$ the results are even more scarce. In \cite{Ass2} , for $H=2$, derived from the ideas in \cite{Ass3}, a class of commuting measure transformations strictly containing K- actions ( for the group generated by $T_1$ and $T_2$) was shown to provide pointwise convergent averages. In \cite{Leib} Leibman showed that once restricted to actions on nilsystems the averages converge.  An approach using random sequences is done by N. Frantzikinakis, E. Lesigne and M. Wierdl \cite{FLW}.\\
    For the norm convergence the situation is pretty much settled. In their initial work, J.P. Conze and E. Lesigne \cite{CL} proved the norm convergence of the averages $\dis \frac{1}{N}\sum_{n=1}^N f\circ T_1^n f_2\circ T_2^n$ for commuting measure preserving transformations $T_1$ and $T_2$ on the same probability measure space. In \cite{HK}, B. Host and B. Kra  and independently T. Ziegler \cite{Z}, proved the norm convergence of the averages $\dis \frac{1}{N}\sum_{n=1}^Nf_1\circ T^n f_2\circ T^{2n}\cdots f_H\circ T^{Hn}.$
   %In \cite{HK} the problem of the pointwise convergence of these averages was raised again.
    In \cite{Tao}, T. Tao extended their result by proving that for commuting measure preserving transformations $T_i,$ $1\leq i\leq H$ on the same probability measure space the averages $\dis \frac{1}{N}\sum_{n=1}^N f_1\circ T_1^nf_2\circ T_2^n\cdots f_H\circ T_H^n$ converge in norm for every bounded function $f_i$ $1\leq i\leq H.$  Another proof was given by T. Austin in \cite{Austin}, B. Host \cite{Host} and 
 H. Townser \cite{Townser}. M. Walsh \cite{Walsh} extended Tao's result to the case where the maps $T_i,$ $1\leq i\leq H$ generate a nilpotent group. In view of the negative result provided by V. Bergelson and A. Leibman  \cite{BL} for solvable groups this is the best possible case for convergence in norm. T.  Austin gave a proof of Walsh `s result using couplings in \cite{Austin2}.

     Our goal is to present a new approach to the pointwise convergence of these non conventional ergodic averages. We do not use the notion of characteristic factors or uniform Wiener Wintner estimates as done in Bourgain `s paper. Our approach will enable us to prove the following results.
     %\begin{theorem}\label{MET1}
     % Let $(X, \mathcal{A}, \mu)$ be a probability measure space and let $T_1$ and $T_2$ be two invertible measure preserving transformations on $(X, %\mathcal{A}, \mu).$ If one of the maps $T_1\circ T_2^{-1}$ or $T_2\circ T_1^{-1}$ is ergodic then the averages
     % $\dis \frac{1}{N}\sum_{n=1}^N f_1(T_1^nx)f_2(T_2^nx) $ converge a.e. for every function $f_1,$ $f_2$ in $L^{\infty}(\mu).$
     %\end{theorem}
  %Several proofs of T. Tao result have been obtained since; \cite{Austin}, \cite{Host}, \cite{Townser}.
   %This brought up the natural question of the pointwise convergence of the averages $$\frac{1}{N}\sum_{n=1}^N \prod_{i=1}^H f_i(T_i^nx)$$ under the commutativity assumption of the transformations. Our goal is to prove this pointwise convergence result in two important cases.
  %which solves the long open problem of the pointwise convergence of the averages $\dis \frac{1}{N}\sum_{n=1}^N \prod_{i=1}^H f_i(T_i^nx).$ It extends the partial results obtained in this direction in \cite{Ass3} (for $I=2$ and $K$-actions), in \cite{Ber}, \cite{Le} and \cite{LeRRu}.
 \begin{theorem}\label{MET1} Let $(X, \mathcal{A}, \mu)$ be a probability measure space and let $H$ be a positive integer. Let $T_i,$ $1\leq i\leq H$ be $H$ commuting transformations on $(X, \mathcal{A}, \mu)$ generating a free action. For every bounded functions $f_i,$ $1\leq i\leq H$ the averages $\dis \frac{1}{N}\sum_{n=1}^N \prod_{i=1}^H f_i(T_i^nx)$ converge a.e.
\end{theorem}
 %We also have the following which answers B. Host and B. Kra question and at the same time provides a different proof of Bourgain a.e. double recurrence theorem.
 As a corollary we have the following result
\begin{theorem}\label{MET2}
Let $(X, \mathcal{A}, \mu)$ be a probability measure space and let $H$ be a positive integer. Let $T$ be an invertible measure preserving transformation acting on $(X, \mathcal{A}, \mu)$ and $f_i$, $1\leq i\leq H$ be $H$ functions in $L^{\infty}(\mu).$
Then the averages $\dis \frac{1}{N}\sum_{n=1}^N \prod_{i=1}^H f_i(T^{in}x)$ converge a.e.
\end{theorem}
   Theorem \ref{MET2} solves a long open problem on the pointwise convergence of nonconventional ergodic averages. At the same time it provides another proof of J. Bourgain's a.e. double recurrence theorem.  \\
   Theorems \ref{MET1} and \ref{MET2} are consequences of the following theorem on finitely many commuting measure preserving homeomorphisms generating a free action.
%To state it we need some

   \begin{theorem}\label{MET3}
     Let $X$ be a compact metrizable space and let $\mathcal{A}$ its Borelian $\sigma$-algebra. Let $T_i,$ $1\leq i\leq H,$ be commuting homeomorphims on $X$ each preserving the same Borel measure $\mu$ and generating a free and minimal $\Z^H$ action. ( in particular we have $\mu^H(O)>0$ for each non empty open set of $X^H$) \\
     We denote by $\Phi: X^H \rightarrow X^H $ the homeomorphism given by the equation $\Phi(z) = (T_1z_1, T_2z_2, ..., T_Hz_H)$ and by $\mu_{\Delta}$ the diagonal measure on $(X^H, \mathcal{A}^H)$ and by $\nu$ the probability measure defined on $(X^H, \mathcal{A}^H)$ by
     $$\nu(A) =\frac{1}{3} \sum_{n=-\infty}^{\infty} \frac{1}{2^{|n|}}\mu_{\Delta}(\Phi^{-n}(A)).$$ Let $F$ be a function defined on $X^H$ of the form
     $\otimes f_i,$ where $f_i\in L^{\infty}(\mu)$ and $1\leq i\leq H.$ \\
    % \begin{enumerate}
     %\item For every function $F\in \mathcal{L}$ defined on $X^H$ the averages \\
  Then  the averages
      $M_N(F) (z) = \frac{1}{N}\sum_{k=1}^N F\circ\Phi^k(z)$ converge $\nu$ a.e.
    % \item If $T_i= T^i,$ $T$ is ergodic and one of the fumction $f_i$ is orthogonal to the  $H$th-Host-Kra-Ziegler factor then $\lim_N M_N(F)(z)= 0\,\, \text{for $\nu$ a.e}$
     %\item If $T_i$ is weakly mixing and one of the functions $f_i$ as zero integral then $\lim_N M_N(F)(z)= 0\,\, \text{for $\nu$ a.e}.$
     %\end{enumerate}
     % Assume that
     % \begin{equation}\label{E1}
     % \nu\{z \in X^H: \{\Phi^nz; n\in \Z\}\, \text{ is dense in } X^H\}= 1.
     % \end{equation}
     % Then for every function $f_i\in L^{\infty}(\mu),$ $1\leq i\leq H,$ the averages
     % $$\frac{1}{N}\sum_{n=1}^N f_1(T_1^nx)f_2(T_2^nx)...f_H(T_H^nx)$$ converge $\mu$ a.e.
   \end{theorem}
     We first give a proof of Theorem \ref{MET3}. Then we will show how to derive from it Theorem \ref{MET2} and Theorem \ref{MET1} .
%\noindent {\bf Remark}
 % Theorem \ref{MET2} was announced in \cite{Ass4}. The proof was not complete but several ideas in it as well as those on a later paper in circulation are used in this current paper.\\
  %and
  %\begin{theorem}\label{MET2}
  % Let $T_i$, $1\leq i\leq H$ , be $H$ measure preserving commuting homeomorphisms on $(X,\mathcal{A},\mu)$ where $X$ is compact and metrizable. If the homeomorphism $\Phi$ defined on $X^H$ as $\Phi (z_1,Z_2,...z_H)= (T_1(z_1), T_2(z_2), ..., T_H(z_H))$ is minimal then the averages
  % $$\frac{1}{N}\sum_{n=1}^N \prod_{i=1}^H f_i(T_i^nx)$$ converge a.e for any bounded functions $f_i,$ $1\leq i\leq H.$
  %\end{theorem}
   %We can observe that Theorem \ref{MET1} gives also another proof of J. Bourgain a.e. double recurrence theorem for the averages
   %$\dis \frac{1}{N}\sum_{n=1}^N f_1(T^nx)f_2(T^{2n}x)$  by taking $T$ ergodic , $T_1 = T$ and $T_2 = T^2$. Assuming $T$ ergodic is not a restriction as we can %use the ergodic decomposition.\\
   The freeness assumption of the action of the group $\mathcal{T}$ generated by the maps $T_i$, $1\leq i \leq H$ made in Theorem \ref{MET1} is not a real restriction for the general commuting case. Indeed for the general case of simply commuting invertible transformations, one can split the space $X$ into two measurable subsets $F_1$ and $F_2$  invariant with respect to the action of the group $\mathcal{T}$ ; restrited to $F_1$ the action of this group is free while on the second the a.e convergence we seek can be obtained in a simple way. (see the last remark at the end of the paper). Based on this remark we will focus mainly on commuting transformations generating free actions.
   
 \noindent{\bf Acknowledgments} . This paper is  a revised version of a paper submitted for publication on March 2014.
 The author thanks T. Austin and J.P. Conze for their comments on a previous version of this paper.
\section{Proof of Theorem \ref{MET3} for  $H=2$}
\subsection{Description of the main steps of the proof}
The first idea of the proof is to use one of B. Weiss models \cite {Weiss} to be able to assume that $X$ is a compact metric space where the maps $T_i$ become homeomorphisms preserving a probability measure defined on $\mathcal{B}(X)$, the Borelian subsets of $X$, giving a positive measure to any non empty open subset of $X.$   We still denote by $X, \mathcal{A}, \mu)$ this new probability measure space. This idea was already mentioned in \cite{Ass5}. Then we transfer  the pointwise convergence problem from $(X,\mathcal{A}, \mu)$ to $(X^2, \mathcal{B}(X^2), \nu)$ where the measure $\nu$ is defined in the statement of Theorem \ref{MET3}. 
The problem becomes one where we can seek  the pointwise convergence of averages of iterates of a continuous function under the map $\Phi.$ In this set up here are the main steps.
\begin{enumerate}
\item
We transfer the $L^1(\mu)$ norm convergence of the averages $\frac{1}{N}\sum_{n=1}^N f_1\circ T_1^n f_2\circ T_2^n$ to the norm convergence of the same averages in $L^1(\nu).$
The map $\Phi$ is nonsingular with respect to $\nu$ and the operator $T$ defined  by
$TF = F\circ \Phi$ is bounded on $L^1(\nu)$ and a contraction on $L^{\infty}(\nu)$. Therefore this operator admits an adjoint $T^*,$   bounded operator on $L^1(\nu)$ and $L^{\infty}(\nu).$
\item We refine the study of the  pointwise convergence of the map $\Phi$ by using Choquet's theorem to obtain a disintegration of $\nu$ into ergodic measures $\nu_m$.  Each measure $\nu_m$ keeps the main properties of $\nu$ such as nonsingularity for $\Phi$, existence of an adjoint operator.  Furthermore because of their ergodicity  the invariant functions such that the limsup and liminf of the averages $M_N(F)(z) = \frac{1}{N}\sum_{n=1}^N F(\Phi^nz)$ are $\nu_m$ a.e. constant.
\item
To overcome the difficulty with the loss of control of the limit of the averages for characteristic functions of open sets we seek open sets with  boundary of measure zero with respect to a given positive measure.
This search is done through the distribution function of continuous function with respect to this positive measure.
%The minimality of the action of $T_1$ and $T_2$ on $(X,\mathcal{A}, \mu)$ indicates that the action of the group generated by the map $\Phi$ and the maps $T_{\gamma}\times T_{\gamma} $  (where
%$T_{\gamma}$ is one of the elements $T_1^p\circ T_2^q$)  on the set $\mathcal{F}$ closure of t%e orbit of any given point $(x,x)$ on the diagonal of $X^2$ is also minimal. This set $\mathcal{F}$ happens to be independent of the point $(x,x)$.
\item
To make the arguments clearer we look first at the case where the averages $M_n(F)$ converge in norm  to the product of the integral of the continuous  functions  $f_i$ with respect to $\mu.$ This happens for instance when the action generated by the maps $T_i$ is weakly mixing, case attempted in \cite{Ass5} then in \cite{Abd}.
%In this case we look at the connection between the positivity of the measure of some opWe use  the discretized variational inequality obtained in \cite{....} to the setting of the dynamical system $(X^2, \mathcal{B}(X^2), \nu, \Phi).$ After studying 
\item
Then we look at the general case . We combine a discretized variational inequality,  with a  consequence of the Furstenberg-Katznelson  theorem \cite{FK} and a reasoning by contradiction to obtain the pointwise convergence of the averages.
 \end{enumerate}

\subsection{Convergence in $L^1(\mu)$ norm of the averages $M_N(F)$}
 We assume that $X$ is a compact metric space,  $\mathcal{A}$ is the set of Borelian subsets of $X$ and $T_1$ and $T_2$ are commuting homeomorphisms on $X$ preserving the same measure $\mu$ on $\mathcal{A}.$ We denote by $\Delta$ the diagonal of $X^2$ i.e. $\{ (x,x)\in X^2 : x \in X\}.$
 %We assume that for any non empty subset $O$ of $X^2$ $\mu(O)>0.$

    We consider now the diagonal measure $\mu_{\Delta}$ as the unique measure defined on $(X\times X, \mathcal{A}^2)$ by the equation\\
    $\dis \mu_{\Delta}(A) = \int \mathbf{1}_A(x, y) d\mu_{\Delta} = \int \mathbf{1}_A(x,x)d\mu$ for any measurable subset $A$ of $X\times X.$
    In particular we have for each measurable function $f$ and $g,$
    \begin{equation}\label{Eq2}
    \int f(x)g(x)d\mu = \int f(x)g(y)d\mu_{\Delta}.
    \end{equation}
   We denote by $\mathcal{L}$ the algebra of finite linear combinations of product functions $f_i\otimes g_i$ defined on $X\times X$ where $f_i$ and $g_i$ are bounded and measurable on $X.$ %We consider two invertible measure preserving transformations $T_1$ and $T_2$ on $(X, \mathcal{A}, \mu)$ .
   The norm convergence result for two commuting measure preserving transformations gives us an operator $R$ defined on $\mathcal{L}$ such that for all function $F\in \mathcal{L}$ and for all measurable subset $W\in \mathcal{A}^2,$
   \begin{equation}\label{eqP}
   \begin{aligned}
    &\lim_L \int \mathbf{1}_{W}(x,x) \frac{1}{L}\sum_{l=0}^{L-1} F(T_1^nx, T_2^nx)d\mu = \lim_L \int \mathbf{1}_{W}(x,y) \frac{1}{L}\sum_{l=0}^{L-1} F(T_1^nx, T_2^ny)d\mu_{\Delta}\\
    &= \int \mathbf{1}_W(x,y)R(F)(x,y)d\mu_{\Delta} .
   \end{aligned}
   \end{equation}
    More can be said about the limit function $R.$
  \begin{lemma}\label{L1}
   For any two invertible commuting measure preserving transformations, $T_1$ and $T_2$ on the probability measure space $(X, \mathcal{A}, \mu)$ and any two  $L^{\infty}(\mu)$ functions, $f_1$ and $f_2$, let us denote by $R(f_1\otimes f_2)$ the norm limit of the averages $$\frac{1}{N}\sum_{n=0}^{N-1}f_1\circ T_1^nf_2\circ T_2^n .$$ If $\mathcal{I}$ is the $\sigma$-algebra of the invariant sets for the measure transformation $T_1\circ T_2^{-1}$ we have
   $$\lim_N\int \big(\frac{1}{N}\sum_{n=0}^{N-1}f_1(T_1^nx)f_2(T_2^nx)\big) d\mu(x)= \int \mathbb{E}[f_1|\mathcal{I}]\mathbb{E}[f_2|\mathcal{I}]d\mu.$$
   So there exists a measure $\omega$ on $(X\times X, \mathcal{A}^2)$ defined by
   \begin{equation}\label{Eq3}
   \omega(f_1\otimes f_2) = \int \mathbb{E}[f_1|\mathcal{I}]\mathbb{E}[f_2|\mathcal{I}]d\mu= \int R(f_1\otimes f_2)(x,y) d\mu_{\Delta}.
  \end{equation}
   In particular if $T_1\circ T_2^{-1}$ or $T_2\circ T_1^{-1}$ is ergodic then $\omega= \mu\otimes \mu.$
  \end{lemma}
  \begin{proof}
  This follows from the commuting property of the transformations $T_1$ and $T_2$ and the mean ergodic theorem as the limit is equal to
  $$\lim_N\int f_1(x)\frac{1}{N}\sum_{n=0}^{N-1}f_2(T_2\circ T_1^{-1})^n(x) d\mu = \int f_1 \mathbb{E}(f_2| \mathcal{I})d\mu.$$
  where $\mathcal{I}$ is the $\sigma$-algebra of invariant subsets of $\mathcal{A}$ for the transformation $ T_2\circ T_1^{-1}.$
  %If $f_1$ and $f_2$ are less than or equal to $1$ then $f_1(x) \geq f_1(x).f_2(x)$ and $f_2(x)\geq f_1(x)f_2(x).$
 The equation
  $$\omega(f_1\otimes f_2) = \int \mathbb{E}[f_1|\mathcal{I}]\mathbb{E}[f_2|\mathcal{I}]d\mu$$ easily defines a measure on $(X\times X, \mathcal{A}^2).$
  The remaining part of the lemma follows directly from the equation (\ref{eqP}).
  %\int f_1(x)\mathbb{E}[f_2|\mathcal{I}]d\mu \geq \int \big(\mathbb{E}[ f_1.f_2|\mathcal{I}]\big)^2d\mu\geq \big(\int f_1(x)f_2(x)d\mu\big)^2.$$
   Finally, if $T_1\circ T_2^{-1}$ is ergodic then the conditional expectations
 $\mathbb{E}[f_1|\mathcal{I}]$ and $ \mathbb{E}[f_2|\mathcal{I}]$ are respectively the integral of $f_1$ and $f_2$ with respect to the measure $\mu.$
 The equality $\omega = \mu\otimes \mu$ follows easily from this last remark.
  \end{proof}
 In the setting we defined above, Lemma \ref{L1} applies to continuous function $F$ defined on $X^2.$  We have the following relations
 \begin{equation}
  \int F(x,y) d\omega =  \lim_L \int\frac{1}{L}\sum_{l=0}^{L-1} F(T_1^lx, T_2^lx) d\mu.
 \end{equation}
 Furthermore for any open subset $O$ of $X^2$ we have
 \begin{equation}
  \int \mathbf{1}_O(x,y) d\omega \leq \liminf_L \int \frac{1}{L}\sum_{l=0}^{L-1}\mathbf{1}_O(T_1^lx, T_2^lx) d\mu.
  \end{equation}
  The last equation follows from the fact that the characteristic function of an open set is an increasing limit of continuous functions.
  Similarly we have for any closed subset $K$ of $X^2,$
  \begin{equation}
  \limsup_L \int \frac{1}{L}\sum_{l=0}^{L-1}\mathbf{1}_K(T_1^lx, T_2^lx) d\mu \leq \int \mathbf{1}_K(x,y) d\omega.
  \end{equation}
  \subsection{Transfer of the norm convergence from $L^1(X, \mathcal{A}, \mu)$ to $L^1(X^2, \mathcal{A}^2, \nu)$}
   From now on we fix  $f_1$ and $f_2$ two bounded continuous real valued functions on $X$. We denote by $F$ the function defined on $X^2$ as $f_1\otimes f_2$ and by $M_L(F)(z)$ the averages $\frac{1}{L}\sum_{n=0}^{L-1} F(T_1^nx, T_2^ny)= \frac{1}{L}\sum_{n=0}^{L-1} F(\Phi^n(z))$ where $z = (x, y)$ and $\Phi^n(z)=(T_1^nx, T_2^ny).$
   Our main goal is to transfer the problem of the pointwise convergence of the averages $\frac{1}{L}\sum_{n=0}^{L-1} f_1(T_1^nx)f_2(T_2^nx)$ with respect to $\mu$ to the one on $X^2$ for the averages $M_L(F)(z)$ with respect to a probability measure on $(X^2, \mathcal{A}^2)$ for which $\Phi$ is nonsingular. To this end we start with the diagonal measure $\mu_{\Delta}$ and introduce the measure $\nu: \mathcal{A}^2 \rightarrow [0,1]$ where
   $\dis \nu(A) = \frac{1}{3}\sum_{n=-\infty}^{\infty}\frac{1}{2^{|n|}} \mu_{\Delta} (\Phi^{-n}(A)).$ It is simple to check that $\nu(\Phi^{-1}(A)) \leq 2\nu(A),$  This property
   %$if $\nu(A)= 0$ then $\nu(\Phi^{-1}(A)) = 0$
    makes $\Phi$ nonsingular with respect to $\nu.$\\
    As a consequence the operator $T:F \rightarrow F\circ \Phi$ is well defined and bounded on $L^1(\nu).$
    We can prove the first part of Theorem \ref{MET3}. We state it in a more general form by considering the set of finite linear combinations of functions of the form $f\otimes g.$
       \begin{lemma}\label{L2}
      Let $\mathcal{L}$ be the set of finite linear combinations of functions of the form $f_1\otimes f_2$ where $f_i$ are bounded and $\mathcal{A}$ measurable.
      Then for each $F\in \mathcal{L}$ the averages
      $\dis \frac{1}{L}\sum_{l=0}^{L-1} F\circ \Phi^l$ converge to a function  $R_{\nu}(F)$ in $L^1(\nu)$ norm. The sets $\Phi^n(\Delta)$ are pairwise disjoint and

      $\dis R_{\nu}(F) = \sum_{n=-\infty}^{\infty} R(F)\circ \Phi^{-n}\mathbf{1}_{\Phi^n(\Delta)}.$ As a consequence we have $\int R_{\nu}(F) d\nu = \int R(F) d\mu_{\Delta}.$
     \end{lemma}
     \begin{proof}
     First let us see quickly why under the assumptions made on the group generated by $T_1$ and $T_2$ the sets $\Phi^k(\Delta)$ and $\Phi^j(\Delta)$ are pairwise disjoint $\nu$ a.e. if $k\neq j.$ We just need to show that all $k', j'\in \Z, k'\neq j'$
     $\dis \mu_{\Delta}(\Phi^{k'}(\Delta) \cap \Phi^{j'}(\Delta))= 0.$ 
     Indeed this will automatically imply 
     %by the nonsingularity of $\Phi$
      that for each $n\in \N$
     $\dis \mu_{\Delta}(\Phi^{-n}((\Phi^k(\Delta) \cap \Phi^j(\Delta))= \dis \mu_{\Delta}(\Phi^{k-n}(\Delta) \cap \Phi^{j-n}(\Delta)) = 0$ by taking $k' = k-n, j' = j-n.$
      And from this we derive easily that $\nu (\Phi^k(\Delta) \cap \Phi^j(\Delta))= 0.$
      But the set $\{ (z,z) \in \Delta: (z,z) \in \Phi^n(\Delta)\cap \Phi^j(\Delta)\}$ is a subset of the set of $\{ (z,z): z = U^n (z)\,\, \text{where} \,\,U
       = T_1\circ T_2^{-1}\}.$ By the freeness of the action of the group generated by $T_1$ and $T_2$ this last set has measure zero with respect to $\mu_{\Delta}.$
      Now we can prove this lemma.
      We know that the averages $\dis \frac{1}{L}\sum_{l=0}^{L-1} F\circ \Phi^l$ converge in $L^1(\mu_{\Delta})$ norm to $R(F).$
       For each $n\in \Z$ we define the measure $\mu_{\Delta}^{(n)}: A \in \mathcal{A}^2 \rightarrow \mu_{\Delta}^{(n)}(A) = \int
       \mathbf{1}_{A}\circ \Phi^n d\mu_{\Delta}.$
       As $F$ is bounded, for each $n\in \Z$ the averages
       $\dis \frac{1}{L}\sum_{l=0}^{L-1} F\circ \Phi^{l+n}$ converge in $L^1(\mu_{\Delta})$ norm to $R(F).$ Therefore we have
       $R(F\circ \Phi^{n})= R(F)$ and $\dis \frac{1}{L}\sum_{l=0}^{L-1} F\circ \Phi^{l}$ converges in $L^1(\mu_{\Delta}^{(n)})$ norm to $R(F)\circ \Phi^{-n}.$
         As the sets $\Phi^n(\Delta)$ are disjoint we get
         $$ \int \left|\frac{1}{L}\sum_{l=0}^{L-1} F\circ \Phi^{l}- R_{\nu}(F) \right|d\nu = \frac{1}{3}\sum_{n=-\infty}^{\infty}\int \frac{1}{2^{|n|}}\left|\frac{1}{L}\sum_{l=0}^{L-1} F\circ \Phi^{l}- R(F)\circ \Phi^{-n}\mathbf{1}_{\Phi^n(\Delta)}\right|d\mu_{\Delta}^{(n)}$$
         $$= \sum_{n=-\infty}^{\infty}\int \frac{1}{2^{|n|}}\left|\frac{1}{L}\sum_{l=0}^{L-1} F\circ \Phi^{l+n}- R(F)\right|d\mu_{\Delta}$$
         As the tail of the series converges uniformly to zero
          $$\big(\text{i.e.}\, \lim_N\sup_{F\in\mathcal{L},\|F\|_{\infty}\leq 1} \sum_{|n|\geq N}\int \frac{1}{2^{|n|}}\left|\frac{1}{L}\sum_{l=0}^{L-1} F\circ \Phi^{l+n}- R(F)\right|d\mu_{\Delta} = 0\big)$$
         we obtain the convergence to zero of this last term when $L$ tends to $\infty.$\\
         The last part of the lemma follows by integration.

     \end{proof}
   \subsection{ Disintegration of $\nu$  into ergodic measures $\nu_m$}  
     We define the set\\
   $\dis \mathcal{M}_2(X^2, \Phi) =
   \{\text{Probability measures}\, \gamma \,\text{on}\, \mathcal{A}^2; \gamma(\Phi^{-1}(A)) \leq 2\gamma(A)\, \text{for all}\,\, A\in \mathcal{A}^2\}.$
   The measures in $\mathcal{M}_2(X^2, \Phi)$ not being necessarily measure preserving, we naturally extend to this setting the notion of ergodicity.
    \begin{mydef}
    A measure $\gamma$ is ergodic if for each  $\Phi$ invariant subset $E$ of $X^2$ we have $\gamma(E) = 1$ or $\gamma(E)= 0.$
    \end{mydef}
   The next lemma lists some properties of this set that will be useful in our proof.
   \begin{lemma}\label{L3}
   The set $\mathcal{M}_2(X^2, \Phi)$ is compact and convex. If a measure $\gamma$ is an extreme point of $\mathcal{M}_2(X^2, \Phi)$ then $\gamma$ is ergodic.
   \end{lemma}
   \begin{proof}
   The facts that $\mathcal{M}_2(X^2, \Phi)$ is bounded and convex are easy to check. Is is also closed because if $\gamma_n$ is a sequence in this set which converges weakly to $\gamma$
    then for all continuous function $F$ on $X^2$ we have

    $\dis \int F\circ \Phi d\gamma = \lim_n \int F\circ \Phi d\gamma_n \leq 2 \lim_n \int F d\gamma_n = 2 \int F d\gamma.$

    This proves the first part of this lemma.
     A proof of the second part follows the lines of the measure preserving case. We give it here for the sake of completeness.  Assume that $\gamma$ is not ergodic. Then we can find an invariant set $E$ such that $0 < \gamma(E) < 1.$ We can define two measures
    $\gamma_1: B\in \mathcal{A}^2\rightarrow  \gamma_1(B) = \frac{\gamma(B\cap E)}{\gamma(E)} $
     and
      $\gamma_2: B\in \mathcal{A}^2\rightarrow  \gamma_2(B) = \frac{\gamma(B\cap E^c)}{\gamma(E^c)}.$
 As $\gamma_1(E) = 1 \neq 0 = \gamma_2(E)$ these two measures are not equal. Furthermore
the equations
$\gamma_1(\Phi^{-1}(B)) = \frac{\gamma(\Phi^{-1}(B)\cap E)}{\gamma(E)}=\frac{\gamma(\Phi^{-1}(B\cap E))}{\gamma(E)}\leq 2 \frac{\gamma(B\cap E)}{\gamma(E)}= 2\gamma_1(B)$ show that $\gamma_1$ and by the same argument $\gamma_2$ are in $\mathcal{M}_2(X^2, \Phi).$ As we have
$\gamma(B) = \gamma(E) \gamma_1(B) + \gamma(E^c)\gamma_2(B)$  the measure $\gamma$ would not be an extreme point of $\mathcal{M}_2(X^2, \Phi).$ A contradiction which shows that $\gamma$ is ergodic.
\end{proof}

   \begin{lemma}\label{L4}
   Let us denote by $\mathcal{E}(X^2, \Phi)$ the set of extreme points of $\mathcal{M}_2(X^2, \Phi).$ Let $\nu$ be the measure defined on $\mathcal{A}^2$ by the formula $\dis \nu(A) = \frac{1}{3}\sum_{n= -\infty}^{\infty} \frac{ \mu_{\Delta}(\Phi^{-n}(A))}{2^{|n|}}.$ Then there is a unique probability measure $\tau$ on the Borel subsets of the compact metrizable space $\mathcal{M}_2(X^2, \Phi)$ such that for all $F$ continuous function on $X^2$ we have
    \begin{equation}
    \int_{X^2} F d\nu = \int_{\mathcal{E}(X^2, \Phi)}\left(\int_{X^2} F d\nu_m(z) \right)d\tau(m).
    \end{equation}
   \end{lemma}
   \begin{proof}
    This is a direct consequence of Choquet representation theorem, see \cite{Walters} for instance. Note that each measure $\nu_m$ in this disintegration of the measure $\nu$ is ergodic by
     Lemma \ref{L2}. Furthermore $\nu_m(\Phi^{-1}(A)) \leq 2\nu_m(A)$ for each $A\in \mathcal{A}^2.$
     %Furthermore for each
   \end{proof}

   %To complete the "transfer" we want to show that the norm convergence of the averages $\frac{1}{L}\sum_{l=0}^{L-1} (g_1\otimes g_2)\circ \Phi^l$ established in \cite{Tao} for the measure $\mu_{\Delta}$ and any bounded $\mathcal{A}$ measurable functions $g_1$ and $g_2$ holds also for the measure $\nu.$ This is the content of the next lemma. We recall that we assume that $T_1\circ T_2^{-1}$ is ergodic and that we denote by $\Delta$ the diagonal of $X^2.$
   %\begin{lemma}\ref{L2}
    %For any continuous real valued functions $g_1$, $g_2$ defined on $X$ the averages $\frac{1}{L}\sum_{l=0}^{L-1} g_1\otimes g_2\circ \Phi$ converge in $L^2(\nu)$ norm to a limit $R_{\nu}(g_1,g_2)$
   %\end{lemma}
  %\begin{proof}
  %\end{proof}
  
   \subsection{Transfer of the norm convergence of $M_N(F)$ from $L^1(\nu)$ to $L^1(\nu_m)$}
         Now we would like to transfer the norm convergence of the averages $\frac{1}{N}\sum_{n=1}^N F\circ \Phi^n$ from  $L^1(\nu)$ to the norm convergence of the same averages along a given subsequence $n_k$ in $L^1(\nu_m).$
         At the same time we will list some properties of the operator $T$ defined by $TF = F\circ \Phi$ that we will use later.
         \begin{lemma}\label{L5}
           Let $\nu_m$  be the ergodic probability measures obtained by disintegration of the measure $\nu$ in Lemma \ref{L3}.  The operator defined by the equation $TF = F\circ \Phi$
           satisfies the following properties;
           \begin{enumerate}
            \item $\|T\|_{L^1(\nu_m)} \leq 2$ and $\|T\|_{L^{\infty}(\nu_m)}=1,$
             $T$ has an adjoint operator $T^*$ such that $\|T^*\|_{L^1(\nu_m)} = 1$,  $\|T^*\|_{L^{\infty}(\nu_m)}\leq 2$ and 
            $\int g T(F) d\nu_m = \int F T^*(g)d\nu_m$
            \item There exists a measurable subset $\tilde{Z}$ of $\mathcal{E}(X^2, \Phi)$ with full $\tau$ measure and an increasing sequence of integers $n_k$ such that
            for $m\in \tilde{Z}$ for every continuous function $F$ on $X^2$ we have
            $\lim_k \frac{1}{n_k}\sum_{j=1}^{n_k} F(\Phi^jz) = R_{\nu}(F)(z)$ $\nu_m$ a.e. \\
            Therefore there exists a measure $\omega_m$ such that \\
            $\lim_k \int \frac{1}{n_k}\sum_{j=1}^{n_k} F(\Phi^jz) d\nu_m = \int R_{\nu}(F) d\nu_m = \int F d\omega_m.$
            Furthermore for all $F$  bounded and measurable we have
            $\int F d\omega_m d\tau(m) = \int Fd\omega.$
            \item If $D$ is a closed subset of $X^2$ then \\
            $\limsup_k\int \frac{1}{n_k}\sum_{j=1}^{n_k} \mathbf{1}_{D}(\Phi^jz) d\nu_m \leq \omega_m(D)$
         \end{enumerate}
         \end{lemma}
         \begin{proof}
          The first part of this lemma follows from the facts that $T\mathbf{1} = \mathbf{1}$ and $\nu_m( \Phi^{-1}(A)) \leq 2\nu_m(A)$ for all $A\in \mathcal{A}^2.$
          For the second part we take a countable dense (for the uniform topology) sequence of continuous functions $F_n.$ By using the $L^1(\nu)$ norm convergence of the sequence $M_n(F_1)$ to $R_{\nu}(F_1)$ we can find a subsequence of integers $n_k^1$ such that \\
          $\int \int \limsup_k |M_{n_k^1}(F_1) - R_{\nu}(F_1)| d\nu_m d\tau(m)= \int \limsup_k |M_{n_k^1}(F_1) - R_{\nu}(F_1)| d\nu = 0.$ \\
           Therefore we can find a set $\tilde{Z}_1$ such that for each $m$ in this set we have
           $\limsup_k |M_{n_k^1}(F_1) - R_{\nu}(F_1)|= 0$ $\nu_m$ a.e.
           By induction using a diagonal process we can extract a subsequence that we simply call $n_k$ such that
          $\lim_k M_{n_k}(F_N)(z) = R_{\nu}(F_N) (z)$ $\nu_m$ a.e. for each $N.$ By density we obtain the same conclusion for the same subsequence $n_k.$
          By the Riesz representation theorem the map $F\rightarrow \int R_{\nu}(F) d\nu_m$ defines a positive linear functional which is the integral of the function $F$ with respect to a probability measure $\omega_m.$ The last assertion follows by integration with respect to $\tau.$\\
          The third part follows from the fact that the characteristic function of a closed set is the decreasing limit of a sequence of continuous functions.
         \end{proof}

        % \noindent{\bf Remark} \\
          %To the cost of one additional null set and a subsequence of the subsequence $n_k,$ we will have the same conclusion for a given function $F$ in $\mathcal{L}$ of the form
       % $  $f\otimes g$ where $f$ and $g$ are two $L^{\infty}(\mu)$ functions. By the same conclusion we mean that
         %  $ \lim_k \int \frac{1}{n_k}\sum_{j=1}^{n_k} F(\Phi^jz) d\nu_m = \int R_{\nu}(F) d\nu_m = \int F d\omega_m.$ 
           %Note that this implies the same equalities for the functions
           %$F_L = \frac{1}{L}\sum_{j=1}^L F\circ \Phi^l.$\\
           \subsection{Open subsets of $X^2$ associated with $M_N(F)$ with boundary of $\omega_m$ measure zero.}
           
         We recall that the function $F$ is fixed and the product of two continuous functions $f_1$ and $f_2$
      defined on $X.$ One of the main difficulty in dealing with the convergence of the averages $M_N(F)(z)$ is the loss of control one has when dealing with measurable subsets of $\mathcal{A}^2$ which are not of the form $A\times B$ or finite linear combination of such rectangles. Another diffculty is given by the fact that the measures $\nu_m$ and $\omega_m$ can be singular with respect to each other . This is the reason why in this subsection we are concentrating on open sets with boundary of measure $\omega_m$ zero and their relation with the measure of the same sets with $\nu_m.$
      \begin{lemma}\label{L6}
      Given $G $  a nonnegative continuous function and $\rho$ a positive probability measure 
      on $(X^2, \mathcal{A}^2)$
      \begin{enumerate}
      \item 
      there exists a set $B_{G, \rho}$ of real numbers with full Lebesgue measure  such that for each $\lambda \in B_{G, \rho}$,  $\{z\in X^2 : G>\lambda \}$  has boundary of measure $\rho$ equal to zero.
      \item 
      Therefore given  $\omega_m$ there exists a set $\mathcal{B}_{F,\omega_m}$ with full Lebesgue measure such that for each $\lambda$ in this set we have
       \begin{equation}
      % &\nu_m\left({\delta} \{z\in X^2: \sup_{k>j }M_k(F)(z) - M_j(F)(z) >\lambda\}\right)= 0 \, \text{ for each }
      %k , j\in \N \, \\
       \omega_m\left({\text{Bd}}\left\{z\in X^2: \sup_{j\leq k< j' }M_k(F)(z) - M_j(F)(z) >\lambda \right\}\right)= 0 \, \text{ for each }
      k , j, j' \in \N.
      \end{equation}
      \end{enumerate}
      \end{lemma}  
      \begin{proof}
       The first part of the Lemma follows from the equation
       $$\int G(z) d\rho = \int_{0}^{\infty} \rho\{z: G(z)>\lambda\} d\lambda = \int_{0}^{\infty} \rho\{z: G(z)\geq \lambda\} d\lambda $$
       From this last equality we deduce the property 
       $\int _{0}^{\infty}\rho\{z : G(z) = \lambda \} d\lambda= 0.$ \\
       For the second part  we just take the intersection of  the countable sets  $\mathcal{B}_{G, \omega_m}$ obtained for each of the continuous non negative function $G = \sup_{j\leq k < j' }M_k(F) - M_j(F).$
      \end{proof}
       We can list some properties of   
       these open sets with zero $\omega_m$ boundary measure.
       \begin{lemma}\label{L7}
        Consider $O$ an open set in $X^2$ with boundary of measure  $\omega_m$ equal to zero. Then 
        \begin{enumerate}
         \item  $\lim_N  \frac{1}{N}\sum_{n=1}^N \nu_m(\Phi^{-n}(O)) = \omega_m(O)$ and therefore 
          $\lim_N  \frac{1}{N}\sum_{n=1}^N \nu_m(\Phi^{-n}(O^c)) = \omega_m(O^c)$ 
         \item   If we denote by $R_{*}(\mathbb{1}_{O})$ the pointwise limit of the sequence $R_{\nu}(F_j)$ where $F_j$ is an increasing sequence of continuous functions converging pointwise to $\mathbb{1}_O$ then 
          $\int R_{*}(\mathbb{1}_{O})d\nu_m = \omega_m(O).$
          Alternatively if we denote by $R_{**}(\mathbb{1}_{\overline{O}})$  the pointwise limit of the sequence $ R_{\nu}(\Gamma_j)$  where $\Gamma_j$ is a decreasing sequence of continuous functions     converging pointwise to $\mathbb{1}_{\overline{O}}$ then $\int R_{**}(\mathbb{1}_{\overline{O}})d\nu_m 
          = \omega_m(\overline{O}) = \omega_m(O).$ 
          \item $R_{*}(\mathbb{1}_O) = R_{**}(\mathbb{1}_{\overline{O}}).$
         \item In particular $\omega_m(\cup_{n= -\infty}^{\infty} \Phi^n(O))  = 
          \int \sup_N R_{*}\left(\mathbb{1}_{\cup_{n= - N}^{N} \Phi^n(O)}\right) d\nu_m.$
          \item  If $A\times A $ is an open subset of $O$ then $R_{\nu}(A\times A) \leq R_{*} (\mathbb{1}_{O})$
         %\item If $\nu_m(O) > 0 $ then $\omega_m(O)> 0.$
         \end{enumerate} 
       \end{lemma} 
       \begin{proof}
        We have 
        \begin{align*}
        & \omega_m(O) \leq \liminf_N\frac{1}{N}\sum_{n=1}^N \nu_m(\Phi^{-n}(O)) \, \, \text{, because the characteristic function of an open set is lower} \\ 
       & \text{semi-continuous )} \\
        &\leq \limsup_N\frac{1}{N} \sum_{n=1}^N \nu_m\left(\Phi^{-n}(O)\right) ) \leq \limsup_N \frac{1}{N}\sum_{n=1}^N \nu_m(\Phi^{-n}(\overline{O}))\\
         &\leq \omega_m(\overline{O}) \, \text{, (because the characteristic function of a closed set is upper semi continuous)}\\ 
         & = \omega_m(O) \, \text{, because the boundary of $O$ has} \,  \omega_m \, \text{ measure 0.} 
        \end{align*}
         This proves the first part of this lemma. \\
         For the second part we have \\
         $\int R_{*}(\mathbb{1}_{O})d\nu_m = \int \lim_j R_{\nu}(F_j) d\nu_m = \lim_j\int F_j d\omega_m \geq \omega_m(O).$  \\
         On the other hand $\mathbb{1}_{\overline{O}} $ is the decreasing pointwise limit of a sequence of continuous functions $\Gamma_j.$ Therefore we have 
         $\int R_{**}(\mathbb{1}_{\overline{O}})d\nu_m = \lim_j \int R_{\nu} (\Gamma_j) d\nu_m = \lim_j \int 
         \Gamma_j d\omega_m = \int \omega_m(\overline{O}) = \omega_m(O)$ and we get the second part of the lemma.\\
          The third part follows from the fact that $R_{*}(\mathbb{1}_O) \leq R_{**}(\mathbb{1}_{\overline{O}})$ and 
          the equalities $$\omega_m(O) = \int R_{*}(\mathbb{1}_O)d\nu_m = \int R_{**}(\mathbb{1}_{\overline{O}})d\nu_m$$ \\
          For the fourth part we just have to observe that if $O$ has boundary with $\omega_m$ measure equal to zero then for each $N$ the open sets $\cup_{n=-N}^N \Phi^n (O)$ satisfy the same property. Thus we can use part 2 to derive the following  equalities \\
        \begin{align*}    
         & \omega_m(\cup_{n= -\infty}^{\infty} \Phi^n(O)) = \lim_N \omega_m( \cup_{n=-N}^N \Phi^n (O)) \\
         &= \lim_N \int R_{*}\left(\mathbb{1}_{\cup_{n= - N}^{N} \Phi^n(O)}\right) d\nu_m =  \int \sup_N R_{*}\left(\mathbb{1}_{\cup_{n= - N}^{N} \Phi^n(O)}\right) d\nu_m.
         \end{align*}
          For the fifth part the functions $R_{\nu} (\mathbb{1}_{A\times A})$ and $R_{**}(\mathbb{1}_{\overline{O}})$ being invariant under $\Phi$ we just need to prove that for each invariant set 
          $\mathcal{I}$ we have $\int_{\mathcal{I}}R_{\nu} (\mathbb{1}_{A\times A}) d\nu_m \leq 
          \int_ {\mathcal{I}} R_{**}(\mathbb{1}_{\overline{O}}) d\nu_m.$
         We have 
        \begin{align*}
       & \int_{\mathcal{I}}R_{\nu} (\mathbb{1}_{A\times A}) d\nu_m = \lim_N \int  \mathbb{1}_{\mathcal{I}}\left( \frac{1}{N}\sum_{n=1}^N \mathbb{1}_{A\times A} (\Phi^nz)\right)d\nu_m= \lim_N\int \left(\frac{1}{N} \sum_{n=1}^N T^{*n}(\mathbb{1}_{\mathcal{I}})\right) \mathbb{1}_{A\times A} d\nu_m \\
       & \leq \lim_N  \int \left(\frac{1}{N} \sum_{n=1}^N T^{*n}(\mathbb{1}_{\mathcal{I}})\right) {\Gamma_j}d\nu_m , \text{where} \, \Gamma_j  \\
       &\text{is a decreasing sequence of continuous functions converging to} \, \mathbb{1}_{\overline{O}} \\
       & =  \int \mathbb{1}_{\mathcal{I}} R_{\nu}(\Gamma_j) d\nu_m \\
       \end{align*}
       
       Therefore we have
       $$
         \int_{\mathcal{I}}R_{\nu} (\mathbb{1}_{A\times A}) d\nu_m \leq \int \mathbb{1}_{\mathcal{I}} R_{\nu}(\Gamma_j) d\nu_m  \, \text{ for each j} $$
  and so      
  $$     \int_{\mathcal{I}}R_{\nu} (\mathbb{1}_{A\times A}) d\nu_m \leq \int \mathbb{1}_{\mathcal{I}} R_{**}(\mathbb{1}_{\overline{O}})d\nu_m = \int \mathbb{1}_{\mathcal{I}} R_{**}(\mathbb{1}_{O})d\nu_m. $$
       
         This proves that $R_{\nu} (\mathbb{1}_{A\times A})\leq  R_{**}(\mathbb{1}_{O})= R_{*}(\mathbb{1}_{O})$
        % we recall that $\omega_m$  is  $\Phi$ invariant and that $\nu_m$ s ergodic ( see ????).
         %We can observe that the invariant set $\cup_{n=-\infty}^{\infty} \Phi^n(O)$  has $\nu_m$ and $\omega_m$ boundary of measure zero. Indeed 
              
       \end{proof}
       
        \begin{lemma}\label{L8}
            Fix $f_1, f_2$ continuous functions on $X$ and let us denote by $F$ the function $f_1\otimes f_2.$
            Assume that the sequence of functions $M_n(F)(z) $ does not converge $\nu_m$ a.e . Then there exists a sequence of integers $n_k$ satisfying the condition $n_{k+1}> n_k^8$  and a positive real number $\theta \in \mathcal{B}_{F,\omega_m}$ such that
            \begin{enumerate}
            \item $\liminf_k\left( \sup_{n_k\leq n < n_{k+1}} M_{n}(F)(z) - M_{n_k}(F)(z) \right) >\theta$
            \item $\nu_m\{z : \sup_{n_k\leq n <n_{k+1} }M_n(F)(z) - M_{n_k}(F)(z) \leq \theta \} < \frac{1}{3^k} + \frac{1}{4^k}$
            \item  for each $k$  we have $\omega_m\{z : \sup_{n_k\leq n <n_{k+1} }M_n(F)(z) - M_{n_k}(F)(z)= \theta \} = 0$
            \end{enumerate} 
            \end{lemma}
             \begin{proof}
             There exists a sequence $n_k$ which we can assume satisfies the condition $n_{k+1} >n_k^8$ for which the sequence $M_{n_k}(F)$ converges $\nu_m$ a.e. to $\int R_{\nu}(F) d\nu_m.$  Let us assume that we have 
            \begin{equation}
             \lim_k \sup_{n_k\leq n <\infty} M_n(F)(z) - M_{n_k}(F) (z)  = \alpha - \int R_{\nu}(F)  d\nu_m >0
            \end{equation}
            Notice that  $  \lim_k \sup_{n_k\leq n <\infty} M_n(F)(z) - M_{n_k}(F) (z) $ is an invariant function and by the ergodicity of $\nu_m$ it is $\nu_m$ a.e. equal to a constant. We can also observe that by the definition of the $\limsup$  we have for each $k,$ 
            $$\alpha \leq \sup_{n_k\leq n< \infty} M_{n}(F)(z).$$
             Therefore we can write 
            \begin{align*}
             &\sup_{n_k\leq n <\infty} M_n(F)(z) - \int R_{\nu}(F) d\nu_m\\
             & = \left(\sup_{n_k\leq n <\infty} M_n(F)(z) - \sup_{n_k\leq n < n_{k+1}} M_n(F)(z)\right)  + \left(\sup_{n_k\leq n < n_{k+1}} M_n(F)(z) - M_{n_k}(F)(z) \right) \\
             &+ \left(M_{n_k}(F)(z) - \int R_{\nu} (F) d\nu_m \right)\\
             &\geq \alpha - \int R_{\nu}(F) d\nu_m 
            \end{align*}
             By induction we can find a subsequence of the sequence $n_k$ that we still denote by $n_k$ such that 
             \begin{enumerate}
              \item $\int | M_{n_k}(F) - \int R_{\nu}(F) d\nu_m | d\nu_m  \leq \frac{\alpha - \int R_{\nu}(F) d\nu_m }{4^{k+1}}. $ 
             \item For each $k$  we have $$\nu_m \left\{ z': \sup_{n_k\leq n< \infty} M_n(F) (z') - \sup_{n_\leq n< n_{k+1}}M_{n}(F)(z') <  \frac{\alpha - \int R_{\nu}(F) d\nu_m}{4}  \right\} > 1- \frac{1}{3^k} \, \text{ by Egorov theorem.}$$
            % \item $\int | M_{n_k}(F) - \int R_{\nu}(F) d\nu_m | d\nu_m  \leq \frac{\alpha - \int R_{\nu}(F) d\nu_m }{4^{k+1}}. $ 
             \end{enumerate}
           % \sup_{n_k\leq n < \infty }M_n(F)(z) - \int R(F) d\nu_m 
            %We choose $\epsilon$ less than $\frac{\alpha - \int R(F)  d\nu_m}{2}$
           % By Egorov theorem we can find a subsequence of the sequence $n_k$ that we denote again by $n_k$ such that  for each $k$ 
            %\begin{equation}
            % \left| \sup_{n_k\leq n < n_{k+1}} M_n(F)(z)  -  \sup_{n_k\leq n <\infty} M_n(F)(z)\right| < \epsilon,
          %\,  \text{on a set}\, A_k \, \text{with measure}\, \nu_m(A_k) \geq  1 - \frac{1}{3^k}
           %\end{equation} 
           Noticing that the first inequality gives us by Chebychev inequality the condition 
           $$\nu_m\left\{z': | M_{n_k}(F)(z') - \int R_{\nu}(F) d\nu_m | > \frac{\alpha - \int R_{\nu}(F)d\nu_m}{4}\right\}\leq \frac{1}{4^{k}} .$$
            We conclude that  we have for each $k$ 
            \begin{equation}\label{eq1L9}
            \sup_{n_k\leq n < n_{k+1}} M_n(F)(z) - M_{n_k}(F) (z)> \alpha - \int R_{\nu}(F)d\nu_m  - 2.\frac{\alpha - \int R_{\nu}(F) d\nu_m}{4} =  \frac{\alpha - \int R_{\nu}(F)  d\nu_m}{2} 
            \end{equation}
            on a set  $A_k$ with complement  $A_k^c$ with measure $\nu_m(A_k^c) < \frac{1}{3^k} + \frac{1}{4^k}.$  This proves part 2 of the lemma. The set $A =  \cap_{k=1}^{\infty} A_k$ has measure greater or equal than $1 - \sum_{k=1}^ {\infty} \left(\frac{1}{3^k} + \frac{1}{4^k} \right) = \frac{1}{6}. $
        As a consequence of (\ref{eq1L9}) we have then on this set $A$ 
        \begin{equation}\label{eq2L9}
        \liminf_k \sup_{n_k\leq n < n_{k+1}} M_n(F)(z) - M_{n_k}(F) (z)\geq  \frac{\alpha - \int R{\nu}(F) d\nu_m}{2} 
            \end{equation}
      But the measure $\nu_m$ being ergodic and the function  $ \liminf_k \sup_{n_k\leq n < n_{k+1}} M_n(F)(z) - M_{n_k}(F) (z)$ invariant,  the inequality (\ref{eq2L9}) holds then for $\nu_m$ a.e. $z$ in $X^2.$
      
      To conclude the proof and establish part three, we just need to pick $\theta>0$ less than $\frac{\alpha - \int R_{\nu}(F)  d\nu_m}{2}$
      in the set $\mathcal{B}_{F,\omega_m}$ defined in Lemma \ref{L6}. Such a value of $\theta$  automatically gives also  part 1 and 2 of the lemma.

        \subsection{A stopping time with the $\liminf_k \sup_{n_k\leq n < n_{k+1}} M_n(F)(z) - M_{n_k}(F)(z)$}
       
         \begin{lemma}\label{L9}
          Assume that $\liminf_k \sup_{n_k\leq n < n_{k+1}} M_n(F)(z) - M_{n_k}(F )(z) > \theta$  for 
          $\nu_m $ a.e. $z$. Define the function  $K: X^2 \rightarrow \N \cup \{\infty\}$ such that
          \begin{enumerate} 
           \item $K(z) = \min\{k: \sup_{n_k\leq n < n_{k+1}} M_n(F) (z) - M_{n_k}(F)(z) > \theta \}$  if the set 
           $\{k: \sup_{n_k\leq n < n_{k+1}} M_n(F) (z) - M_{n_k}(F)(z) > \theta \}$ is not empty.
           \item $K(z) = \infty $ otherwise. 
          \end{enumerate}
          Then there exists $K_0\in \N$ and an open set $O$ with positive $\nu_m$ measure such that 
          on $O$ we have $ K(z) \leq K_0.$ Furthermore  we have $\omega_m\left(\text{Bd} (O)\right) = 0.$
          \end{lemma}
          \begin{proof}
           As $K(z) $ is $\nu_m$ a.e. finite we can find a natural number $K_0$ such that the set 
           $$B = \{z : K(z) \leq K_0\}$$  has $\nu_m$ measure greater than $\frac{2}{3}.$
           The complement of the set $B$ in $X^2$ is the set of $z$ where $K(z) >  K_0.$ 
           This set is included in the closed set $ \cup_{k=1}^{K_0} \{ z: \sup_{n_k \leq n < n_{k+1}} M_n(F)(z) - M_{n_k}(F)(z) \leq \theta\}$ \\
           Its measure is lless than $\frac{5}{6}$ by part 2 of Lemma \ref{L9}. Therefore its complement,  an open set, $O$ included in 
           $B,$ has measure greater than $\frac{1}{6}.$  As each of the sets 
           $\{z: \sup_{n_k\leq n< n_{k+1}} M_n(F(z) - M_{n_k}(F)(z) > \theta\}$  also has boundary  with $\omega_m$ measure equal to zero , the set $O$ has the same property.
          
          \end{proof}
    \subsection{Proof of Theorem \ref{MET3} when the $T_i$ are such that $R(F) = \prod_{i=1}^2 \int f_id\mu$.} 
        
          For a better reading of the arguments, in this subsection we assume that the limit in norm of the averages $M_N(F)$  for $F = f_1\otimes f_2$  is the product  $\int f_1d\mu \int f_2d\mu$ for each continuous function $f_1, f_2$.  In other words the measure $\omega$ is equal to $\mu\otimes \mu.$ Such is the case for instance when the transformation $T_i$  and $T_{i}\circ T_j^{-1}$  for $i\neq j$ are weakly mixing . Another example can be given by two irrational rotations on the one dimensional torus $\T,$  $ T_{i} (x) = x + t_i,  i =1,2 $  with $t_1$ and $t_2$ linearly independent over the rationals. 
          \begin{lemma}\label{L10}
            Assume  that $R(F) = \int f_1 d\mu \int f_2 d\mu $ for each continuous function $f_1, f_2$ defined on $X$ then 
            \begin{enumerate}
            \item $\omega_m = \mu\otimes \mu $  for $\tau$ a.e. $m.$ 
            \item If $O$ is an open set in $X^2$ 
            %with boundary measure  $\omega_m$ equal to zero
             then the  assumption $\nu_m(O) >0$ implies the condition $\omega_m(O) >0.$
            \end{enumerate}
          \end{lemma}
            \begin{proof}
              By assumption we have $R(F) = \int f_1\otimes f_2 d\mu\otimes \mu.$ As a consequence we derive from Lemma \ref{L5} the equation 
              $$\int f_1 \otimes f_2 d\omega_m =  \int \left(\int f_1\otimes f_2 d\mu\otimes \mu \right)d\nu_m 
              = \int f_1\otimes f_2 d\mu\otimes \mu. $$
              From this we get $\omega_m = \mu\otimes \mu.$ \\
              The second part of the lemma follows from the assumption made on $\mu\otimes \mu;$ for every non empty open set $\mu\otimes \mu(O)> 0.$ Therefore if $\nu_m(O) > 0$  then $O$ is not empty and 
              $\mu\otimes \mu(O) > 0.$   
              
              \end{proof}
           
\noindent{\bf End of the proof when $\omega_m= \omega,$ use of a variational inequality} 

        To end the proof in the particular case where $\omega_m = \omega$ we are going to use the following variational inequality by  H.White, see \cite {Ass6}, Theorem 1.9 for instance.  We use the following notations  $M_n(a)(j) = \frac{1}{n}\sum_{l=0}^{n} a_{l+j}.$
        
        For any sequence of natural numbers $n_k$ satisfying the condition $n_{k+1}> n_k^8$ we have  for each sequence $a_j$ of real numbers  and for each $J, K\in \N$
        \begin{equation}\label{Veq}
         \sum_{j= 1}^{J} \sum_{k=1}^K  \left|\sup_{ n_k \leq  n < n_{k+1}} M_n(a) (j) -M_{n_k}(a)(j)\right|^2 \leq C \sum_{j = 1}^{J + n_{K+ 1}} |a_j|^2 
        \end{equation}
        where the constant $C$ is independent of the sequence $a_j,$ $J$ and $K.$ \\
        First we can notice that for any $L\in \N$  we have the inequality  
        \begin{equation}
            \mathbb{1} _{O}L \sum_{k= K_0}^{K_0 + L}\sup_{n_k \leq n <n_{k+1}} M_n(F)(z) - M_{n_k}(F) (z) > \theta   
        \end{equation}
        This is because on $O$ we have $K(z) \leq K_0$ and so  
        $$\inf_{k > K(z) } \sup_{n_k \leq n <n_{k+1}} M_n(F)(z) - M_{n_k}(F) (z) \leq \inf_{ k\geq K_0}
        \sup_{n_k \leq n <n_{k+1}} M_n(F)(z) - M_{n_k}(F) (z)$$
         
         We derive that for each $j, 1\leq j\leq J $  we have 
         \begin{equation}
           \mathbb{1}_{O}(\Phi^jz)  \sum_{k= K_0}^{K_0 + L}\sup_{n_k \leq n <n_{k+1}} M_n(F)(\Phi^jz) - M_{n_k}(F) (\Phi^jz) > L\theta \mathbb{1}_O(\Phi^jz)  
         \end{equation}
          Therefore we get 
          \begin{equation}
        L \theta\sum_{j=1}^J  \mathbb{1}_O(\Phi^jz) \leq \sum_{j=1}^J  \mathbb{1}_O(\Phi^jz)\sum_{k= K_0}^{K_0 + L}\sup_{n_k \leq n <n_{k+1}} M_n(F)(\Phi^jz) - M_{n_k}(F) (\Phi^jz)     
          \end{equation}
        Applying the Cauchy Schwarz inequality we obtain
          \begin{align*}
           &L\theta\sum_{j=1}^J  \mathbb{1}_O(\Phi^jz) \\
           &\leq  \sum_{j=1}^J  \mathbb{1}_O(\Phi^jz)  \sum_{k= K_0}^{K_0 + L}\sup_{n_k \leq n <n_{k+1}} M_n(F)(\Phi^jz) - M_{n_k}(F) (\Phi^jz)  \\
           &\leq\left(\sum_{j=1}^J \mathbb{1}_O(\Phi^jz)\right)^{1/2} \left( \sum_{j=1}^J \left(\sum_{k= K_0}^{K_0 + L}\sup_{n_k \leq n <n_{k+1}} M_n(F)(\Phi^jz) - M_{n_k}(F) (\Phi^jz)\right)^2 \right)^{1/2} \\  
           & \leq \left(\sum_{j=1}^J \mathbb{1}_O(\Phi^jz)\right)^{1/2}\sqrt{L}\left(\sum_{j=1}^J \left(\sum_{k= K_0}^{K_0 + L}\left(\sup_{n_k \leq n <n_{k+1}} M_n(F)(\Phi^jz) - M_{n_k}(F) (\Phi^jz)\right)^2\right)\right)^{1/2}
          \end{align*}
        We conclude that 
        $$\sqrt{L} \theta \left(\sum_{j=1}^J \mathbb{1}_O(\Phi^jz)\right)^{1/2} \leq \left(\sum_{j=1}^J \left(\sum_{k= K_0}^{K_0 + L}\left(\sup_{n_k \leq n <n_{k+1}} M_n(F)(\Phi^jz) - M_{n_k}(F) (\Phi^jz)\right)^2\right)\right)^{1/2}$$
        
        Applying the discrete variational inequality in (\ref{Veq}) to the sequence $a_j = F(\Phi^jz)$ we obtain the inequality
        $$\sqrt{L} \theta \left(\sum_{j=1}^J \mathbb{1}_O(\Phi^jz)\right)^{1/2} \leq C \left(\sum_{j=1}^ {J + n_{K_0 +L +1}} F(\Phi^jz)^2\right)^{1/2}$$
        Squaring both sides , taking their integrals with respect to $\nu_m$ and dividing by $J$ we get 
        $$L \theta^2 \frac{1}{J}\sum_{j=1}^J \nu_m(\Phi^{-j}(O)) \leq C^2 \frac{ J + n_{K_0} +L +1}{J}
       \,   \text{as the function $F$ is bounded by one}.$$

        Then taking the $\limsup $ when $J$ goes to $\infty$ we obtain 
        $${L}\theta^2 (\mu\otimes \mu (O)) \leq C^2 .$$
            Letting $L$ go to infinity we derive a contradiction as $\mu\otimes\mu(O) > 0$  by Lemma \ref{L10} and $\theta>0.$
             This implies that we must have $\limsup_n M_n(F) (z) = \int F d\nu_m.$ 
             By changing $F$ into $-F$ the same argument would give us 
             $$-\liminf_n M_n(F)(z) = \limsup_n M_n(-F)(z)  = \int (-F) d\nu_m = -\int F d\nu_m.$$
              Thus the averages  $M_n(F)(z)$ converge  $\nu_m$ a.e.
              We can conclude by using the disintegration in Lemma  \ref{L10}
              $$\int \left( \limsup_n M_n(F)- \liminf_n M_n(F)\right) d\nu = \int \left( \limsup_n M_n(F)- \liminf_n M_n(F)\right) d\nu_m d\tau(m) = 0 .$$ 
               This means that the averages $M_n(F)$ converge $\nu$ a.e. and therefore $\mu_{\Delta}$ a.e
               as a nullset for $\nu$ is also a nullset for $\mu_{\Delta}.$  
               From this we can conclude that the averages $M_n(F)(x,x)$ converge a.e. $\mu.$ which is what we were looking for.
             
            \end{proof}
            
  %\noindent{\bf Remarks}
  % The proof given in the case where the limit $R (F) = \prod_{i=1}^2 \int f_1d\mu \int f_2 d\mu$ shows that the only part where this assumption is used is part 2 of Lemma ???. It says that if $\nu_m(O) >0 $ then 
  % $\omega_m(O) > 0.$  The measure $\nu_m$ as shown in Lemma  ???? can be discrete and most likely singular with respect to $\omega_m$ when $\omega_m = \mu\otimes \mu.$ For the general case we will focus on this implication.
            
            \subsection{Proof of Theorem \ref{MET3} - General case}
            
 The proof given for the case where the limit $R (F) = \prod_{i=1}^2 \int f_1d\mu \int f_2 d\mu$ for all continuous functions $f_1$ and $f_2$ on $X$ shows one thing ; the only part in the proof where this assumption is used is part 2 of Lemma \ref{L10}. It says that if $\nu_m(O) >0 $ then 
   $\omega_m(O) > 0.$   
  To establish the same property in the general case we need  a more specific disintegration of the measure $\nu$ into ergodic measures $\nu_m.$
  %We can refine this disintegration and be more specific on the measures $\nu_m.$ 
     \begin{lemma}\label{L11}
        Consider for $z\in \Delta$ the measure 
        $\nu_z = \frac{1}{3} \sum_{n= -\infty}^{\infty} \frac{1}{2^{|n|}}\delta _{\Phi^nz}$ where $\delta_{\Phi^nz}$ is the Dirac measure at the point $\Phi^nz.$
        Then $\nu_z$ is ergodic and for each continuous function $F$ on $X^2$ we have 
        $\int F d\nu = \int \int F(z') d\nu_z(z') d\mu_{\Delta}(z).$
        The measures $\nu_z$ satisfy the properties of the measure $\nu_m$ given in Lemma \ref{L5}
     \end{lemma}
     \begin{proof}
         To establish the ergodicity of $\nu_z$  take an invariant set $A.$  If it contains the orbit  
         $\{ \Phi^nz : n\in \Z \}$ then its $\nu_z$  measure is one. If not it is zero. \\
         We have 
         $\int_{X^2} F(z') d\nu_z (z') = \frac{1}{3} \sum_{n=-\infty}^{\infty} \frac{1}{2^{|n|} } F( \Phi^nz) .$ \\
         Therefore we have  
         $\int \int_{X^2} F(z') d\nu_z(z') d\mu_{\Delta}(z) = \int_{X^2} F(z) d\nu(z).$
         It is simple to check that the map $\Phi$ is nonsingular with respect to $\nu_z$ as $\nu_z(\Phi^{-1}(A)) \leq 2 \nu_z(A)$ as well as the other properties of the measures $\nu_m.$

     \end{proof}
     
       \noindent{\bf Remark}
        We denote by $\nu_z$ the measure $\nu_m$ and by $\omega_z$ the corresponding measure $\omega_m$
     
     \begin{lemma}\label{L12}
     There exists a subsequence of integer $N_k$ and a set $\mathcal{Z}$ with full $\mu_{\Delta}$ measure such that for every $z\in
     \mathcal{Z}$ the following holds 
     \begin{enumerate}
     \item for every continuous function $F \in \mathcal{C}(X^2)$ 
     the averages $\frac{1}{N_k}\sum_{n=1}^{N_k} F(\Phi^nz) $ converge to $R_{\nu}(F)(z)= \int F(z') d\omega_z$ 
     \item If $A_i$ is a countable basis for the topology of $X$ then the averages 
                 $\frac{1}{N_k} \sum_{n=1}^{N_k} \mathbb{1}_{A_i \times A_i}(\Phi^nz)$ converge to $R_{\nu}(\mathbb{1}_{A_i\times A_i})(z) = \int  \left(\mathbb{1}_{A_i\times A_i}\right)(z')d\omega_z$
    \item For the same countable basis $(A_i)_i$ and for each $j\in \N$ the averages 
          $\frac{1}{N_k} \sum_{n=1}^{N_k} \mathbb{1}_{(A_i\cap A_j) \times (A_i\cap A_j)}(\Phi^nz)$ converge to \\ $R_{\nu}\left(\mathbb{1}_{(A_i\cap A_j)\times (A_i\cap A_j)}\right)(z) = \int \left(\mathbb{1}_{(A_i\cap   A_j)\times (A_i\cap A_j)}\right)(z') d\omega_z$
  \item If $ E\times E$ is one of the subsets in part 2) or part 3) above we have 
           $\mathbb{1}_{E\times E}(z) R_{\nu} (\mathbb{1}_{E\times E} ) (z) > 0.$ 
    
     \end{enumerate}
      
     \end{lemma}
     \begin{proof}
     Because of the norm convergence of the averages $M_n(F)$ to $R_{\nu}(F)$ established in Lemma \ref{L2} we can extract a susbsequence $ N_k$ along which the averages 
     $\frac{1}{N}\sum_{n=1}^N F\circ \Phi^nz$  converge $\nu$  a.e. to $R_{\nu}(F)(z)$ for each function $F$ in a countable dense set of continuous functions on $X^2.$  The same conclusion follows then for each function in $\mathcal{C}(X^2)$ by approximation. 
      This eliminates a set of measure zero for $\nu$  off which the averages of continuous functions along $N_k$ converge .\\
      As the set $\mathcal{M}= \left\{ \mathbb{1}_{A_i\times A_i}, \mathbb{1}_{A_i\cap A_j \times A_i\cap A_j}, i, j \in \N\right\}$ is countable one can extract a subsequence of $N_k$ that we still denote by $N_k$ along which the averages along the subsequence $N_k$ of functions $G$ in $\mathcal{M}$  will converge $\nu$ a.e. to their limit $R_{\nu}(G)(z).$ By eliminating a null set for $\nu$ we obtain a subset of $X^2,$
 $\mathcal{Z}_1,$  on which the averages of continuous functions and of functions in $\mathcal{M}$ converge along the sequence $N_k$ to their respective limit.\\
 It remains to show that this limit is in fact modulo a nullset equal to the integral of these functions with respect to the measure $\omega_z.$
 %This follows from Lemma ??? part (2) and the ergodicity of the measures $\nu_z$ 
 This follows from Lemma \ref{L11}. We have 
  $$0 = \int \limsup_{k}\left| M_{N_k}(F)(z) - R_{\nu}(F)(z)\right| d\nu = \int \int  \limsup_{k}\left| M_{N_k}(F)(z') - R_{\nu}(F)(z') \right| d\nu_z(z') d\mu_{\Delta}(z)$$
 
 We conclude that we can find a set of full measure $\mathcal{Z}_2$ such that for $z\in 
 \mathcal{Z}_2$  we get 
  $$\lim_k\int \frac{1}{N_k}\sum_{n=1}^{N_k} F(\Phi^nz') d\nu_z = \int R_{\nu}(F)d\nu_z = \int F d\omega_z$$ for each $F\in \mathcal{C}(X^2).$ 
 %We can notice that the function $R_{\nu}(F) $ being $\Phi$ invariant we have $R_{\nu}(F) (z)) = \int R_{\nu}(F) d\nu_z.$ \\
 We can compute   $\int R_{\nu}(F) d\nu_{z}$ by using the definition of the measure $\nu_z$ ( see Lemma \ref{L11})
and the invariance of $R_{\nu}(F) $ under $\Phi$ 
$$\int R_{\nu}(F) (z') d\nu_z = \frac{1}{3}\sum_{n=-\infty}^{\infty} \frac{1}{2^{|n|}}R_{\nu}(F) (\Phi^nz) =
      R_{\nu}(F)(z) .$$
  We can do similar computations for the countable set of functions in $\mathcal{M}.$ This gives us a third set of full measure $\mathcal{Z}_3$ on which the averages of $M_{N_k}(G) (z)$ converge to $\int G d\omega_z$ for each $G\in \mathcal{M}.$ \\
    The last part of this lemma is a consequence of the Furstenberg-Katznelson theorem  \cite{FK}
    They showed that if $V \in\mathcal{A}$ is a set of positive $\mu$ measure then $\int \mathbb{1}_V R(\mathbb{1}_V \times \mathbb{1}_V) d\mu > 0$ \\
    Taking a subset $W$ of $V$ with positive measure  we derive the inequality
     $$ \int \mathbb{1}_W  R(\mathbb{1}_V \times \mathbb{1}_V) d\mu \geq \int \mathbb{1}_W R(\mathbb{1}_W \times \mathbb{1}_W) d\mu > 0.$$   This means that 
      $R(\mathbb{1}_V \times \mathbb{1}_V) >0 $  on $ V.$ 
       Translated to $(X^2,\mathcal{B}(X^2), \nu)$ this inequality becomes
       $$ \mathbb{1}_{E\times E} (z) R_{\nu} ( \mathbb{1}_E \times \mathbb{1}_E )(z) >0, \, \text{ for $\nu$ \, a.e  and therefore for $\mu_{\Delta}$ a.e z} $$
  \\
      Eliminating a null set we obtain a set with full measure $\mathcal{Z}_4$ where for each element $Y$ of $\mathcal{M}$  we get 
      $$ \mathbb{1}_{Y} (z) R_{\nu} ( \mathbb{1}_{Y} )(z) >0 $$   
         
   The set we seek is just $\mathcal{Z} = \cap_{i=1}^4 \mathcal{Z}_i.$
   
      \end{proof} 
      \noindent{\bf End of the proof of the general case for H=2.} \\
        We want to show that if $z \in \mathcal{Z}$ and $\nu_{z}(O) >0$ then $\omega_z(O) > 0$
        where $O$ is an open subset of $X^2$ with boundary of measure $\omega_z$ equal to zero.
        Then we can use the steps in Lemma \ref{L7} to Lemma \ref{L9} and the section with the variational  inequality to conclude that the averages $M_N(F)(z)$ converge pointwise. \\
        So let us assume that $\nu_z(O) > 0.$ Then there exists $n\in\N$ such that $\Phi^nz \in O.$ Thus $z \in O' = \Phi^{-n}(O) $ and we have $\nu_z(O') >0. $
         As $z \in \Delta$ there exists $x\in X$ such that $z= (x, x)$ . The set $O'$ being open there exist two open  subsets $A$ and $B$ of $X$ in $\mathcal{M}$  each containing $x$ such that $A\times B\subset O'.$
          As $x$ belongs to both $A$ and $B$ we have $\nu_z\left( (A\cap B) \times (A\cap B)\right)> 0.$
          We denote by $E$ the set $A\cap B.$ \\
         By part 4 of Lemma \ref{L12} we know that $\mathbb{1}_{E\times E} (z)R_{\nu}( \mathbb{1}_{E\times E})(z) > 0 .$     
        Therefore if $\nu_z(O) >0$ then we get the inequality $\nu_z(E\times E)> 0 . $ This implies that 
        $R_{\nu}( \mathbb{1}_E \times \mathbb{1}_E) (z) > 0.$ By part 5 of Lemma \ref{L7}  this in turn implies that 
        $R_{*}(\mathbb{1}(\Phi^{-n}(O)) (z) >0 $ and  $\omega_z( \Phi^{-n}(O)) > 0.$ As $\omega_z$ is measure preserving this gives the conclusion we seek , i.e. $\omega_z( O) > 0.$  
      
                      \section{Proof of Theorem \ref{MET3} for $H>2$, Theorem \ref{MET1} and Theorem \ref{MET2}}
                      \subsection{Proof of Theorem \ref{MET3} for $H>2$}
                       The proofs of each Lemma extends without difficulty from the case $H=2$ to the case of $H>2$ commuting homeomorphisms .  The map $\Phi$ is just this time 
                       $$\Phi : X^H \rightarrow X^H \, \text{such that} \,\Phi(z) = (T_1(z_1), T_2(z_2), ..., T_H(z_H))$$
                       if $z = ( z_1, z_2, ..., z_H)$. The measure $\mu_{\Delta}$ is the diagonal measure defined by the formula $\int \mathbb{1}_{A}(x,x,...,x) d\mu = \int \mathbb{1}_A(x_1,x_2,..., x_H)d\mu_{\Delta}.$
                       for all measurable set $A\in \mathcal{A}^H.$ The only slightly  different lemma  is the analog of Lemma \ref{L1} for $H$ commuting transformations.
                       \begin{lemma}
                       For any $H$ invertible measure preserving transformation, $T_i, 1\leq i \leq H$ on the probability measure space $(X, \mathcal{A}, \mu)$ and any $H$ functions , $f_1, f_2, ...f_H$ let us denote by $R( f_1\otimes f_2 \otimes...\otimes f_H)$ the norm limit of the averages 
                       $$\frac{1}{N} \sum_{n=1}^N f_1\circ T_1^n f_2\circ T_2^n...f_H\circ T_H^n.$$
                       So there exists a measure $\omega$ on $(X^H, \mathcal{A}^H)$ defined by the formula 
                       $$\omega ( f_1\otimes f_2 \otimes...\otimes f_H) = \lim_N \int \frac{1}{N} \sum_{n=0}^{N-1} 
                       f_1(T_1^nx) ...f_H(T_H^nx) d\mu = \int R( f_1\otimes f_2 \otimes...\otimes f_H)(x_1, ...x_H) d\mu_{\Delta}.$$
                       \end{lemma}
                       
                        \subsection{Proof of Theorem \ref{MET1}}
                     In order to apply Theorem \ref{MET3} we just need to notice that we can put ourselves in the setting of measure preserving homeomorphism by using one of B. Weiss models \cite{Weiss}. Starting with an atomless measure space $(X, \mathcal{A}, \mu)$  and $T_i$ invertible measure preserving transformations generating a free action then there exists an isomorphism allowing to assume that the maps $T_i$ satisfy exactly the assumptions of Theorem \ref{MET3}.  Theorem \ref{MET1} follows from Theorem \ref{MET3} from this isomorphism.  The pointwise convergence of the averages with respect to $\nu$ implies the same conclusion for the measure $\mu_{\Delta}.$  
                     \subsection{Proof of Theorem \ref{MET2}}
                         This time the group is generated by the powers of a single invertible measure preserving transformation $T$ on an atomless probability measure space $(X, \mathcal{A}, \mu).$ We can assume that $T$ is ergodic which makes $T$ aperiodic. Therefore the action of the group generated by this group is free. We can again use Weiss model \cite{Weiss} to put ourselves under the assumptions of Theorem \ref{MET3}. By this theorem and Lemma \ref{L9} we know that if one looks on $(X^H, \mathcal{A}^H, \mu^H)$ at the averages
                         $M_N(F)(z) = \frac{1}{N}\sum_{n=1}^N F\circ\Phi^nz$ where $F = f_1\otimes f_2...\otimes f_H$ where $|f_i|\leq 1,$ then these averages, $M_N(F)(z),$  converge $\nu$ a.e.  But because a null set for $\nu$ is automatically a nullset for $\mu_{\Delta}$ the averages $M_N(F)(z)$  converge $\mu_{\Delta}$ a.e.
                         This means that the averages $\frac{1}{N}\sum_{n=0}^{N-1} f_1(T^nx) ...f_H(T^{Hn}x)$ converge $\mu$ a.e. 
                         
\noindent{\bf Remarks.}

\begin{enumerate}
\item
  As said in the introduction once the free action case solved one can get the general commmuting case in a simple way. For instance for $H=2$ the elements generated by the maps $T_1$ and $T_2$ are in
  $$\mathcal{T} = \{T_1, T_2, T_1\circ T_2^{-1}, T_2\circ T_1^{-1}\} = \{ T_{\gamma}, \gamma \in \Gamma
  \}$$. The split of $X$ can be made with the disjoint sets  
  $$F_1 = \cup_{\gamma\in \Gamma} \cup_{n= - \infty}^{\infty} \{x ; T_{\gamma}^nx = x\} \,\, \text{and } F_2 = F_1^c.$$ It is simple to check that these two sets are invariant under $\mathcal{T}.$ The pointwise convergence of the averages $M_n(f_1\otimes f_2) $ on the set $\{x: T_1^p x =x\}$ can be obtained by first looking at the averages 
  $$\frac{1}{N}\sum_{n=1}^N f_1( T_1^{pn}x) f_2(T_2^{pn}x)= \frac{1}{N}f_1(x)\sum_{n=1}^N f_2(T^{pn}x) \, 
  \, \text{which converge by Birkhoff theorem}.$$ Then we derive the pointwise convergence of the averages 
  $$\frac{1}{N}\sum_{n=1}^N f_1(T_1^{pn +q}x) f_2(T_2^{pn +q}x) $$ by applying the previous pointwise convergence to the functions $(f_1\circ T_1^q)\otimes (f_2\circ T_2^q).$ Regrouping these results for $0\leq q < p,$ we get the convergence of the averages 
  $M_N(f_1\otimes f_2)$ by writing them as $$\frac{1}{N} \sum_{q=0}^{p-1}\sum_{n= 0}^{ [\frac{N}{p}]}  f_1(T_1^{pn +q}x) f_2(T_2^{pn +q}x). $$ 
  The same argument works for each set $\{x; T_2^qx = x\}$ for some $q.$  For a set of the form 
  $\{x : (T_1\circ T_2^{-1})^m x = x\}$ we have then $T_1^mx = T_2^mx.$ The result follows simply from the Birkhoff theorem. 
\item
   Once transferred to the space  $(X^2, \mathcal{A}^2, \nu)$ the commuting assumption of $T_1$ and $T_2$ is only used in the Furstenberg -Katznelson recurrence result. This allows to get some valuable information on the pointwise convergence of the averages $\frac{1}{N} \sum_{n=1}^N f_1(T_1^nx) f_2(T_2^nx)$ when the maps are not necessarily commuting but generate a nilpotent action.  This is studied in \cite{Ass8}.
\item
 Ideas in this paper can be used to prove that the ergodic Hilbert transform, i.e. the series 
 $\sum_{n= - \infty}^{\infty \, `} \frac{\prod_{i=1}^H f_i(T_i^n x)}{n} $ converges a.e.
 This is done in \cite{Ass7}.
 \item 
 The method we use would most likely provide some information on the pointwise convergence of averages 
 of the form $\frac{1}{N}\sum_{n=1}^N f_1(T_1^{p(n)}x)f_2(T_2^{p(n)}x) $ where $p$ is a polynomial with integers values. At the present time we do not know if it would apply directly to averages of the form $\frac{1}{N}\sum_{n=1}^N f_1(T^{n^2}x) f_2(T_2^{n}x).$
 \end{enumerate}

\end{document}